\documentclass[12pt,reqno]{amsproc}
\usepackage{mathrsfs, amsfonts, amsthm, amsmath, amssymb}
\usepackage{enumerate}
\usepackage[shortlabels]{enumitem}
\usepackage{enumitem}

 \theoremstyle{plain}
\newtheorem{theorem}{Theorem}[section]

\newtheorem{lemma}[theorem]{Lemma}

\newtheorem{proposition}[theorem]{Proposition}

\newtheorem*{theorem*} { }

\makeatletter
\newcommand{\LeftEqNo}{\let\veqno\@@leqno}

\makeatother

\numberwithin{equation}  {section}

\begin{document}

\

\vspace{-2cm}

\title{Solvability of $\binom{2k}{k} = \binom{2a}{a} \binom{x+2b}{b}$}

\author{Meaghan Allen}
\address{Meaghan Allen \\
        Department of Mathematics and Statistics \\
         University of New Hampshire\\
         Durham, NH 03824, US; \  Email: Meaghan.Allen@unh.edu   }

\keywords[2020 Mathematics Subject Classification: {11B65, 05A10}


\begin{abstract}
Suppose $k,x,$ and $b$ are positive integers, and $a$ is a nonnegative integer such that $k=a+b$. In this paper, we will prove 
$\binom{2k}{k} = \binom{2a}{a} \binom{x+2b}{b}$ if and only if $x=a=1$. We do this by looking at different cases depending on the values of $x$ and $k$. We use varying techniques to prove the cases, such as direct proof, verification through Maple software, and a proof technique found in Moser's paper. Previous results from Hanson, St\u anic\u a, Shanta, Shorey and Nair are also used.
\end{abstract}

%\subjclass[2010]{Primary 47B49; Secondary 54E40}
%\keywords{ }

\maketitle
%\end{frontmatter}
%\label{}
%\baselineskip18pt

\section{Introduction}
It was discovered by Moser in \cite{Mor} that the equation
$$ \binom{2n}{n} = \binom{2a}{a} \binom{2b}{b}$$
has no solutions. This result was further extended by P. Erdos in \cite{erdos}, where he proved that
$$
\binom{2m}{m} \nmid \binom{2n}{n}
$$ for $2m>n$.
Following the line investigated by Moser in \cite{Mor} and  Erdos in \cite{erdos}, the purpose of this paper is to prove the following result:

\begin{theorem}\label{Thm1.1}
Let $k,x, b$ be positive integers and $a$ be a nonnegative integer such that $k=a+b$. Then
\[ \binom{2k}{k} = \binom{2a}{a} \binom{x+2b}{b}, \qquad \text { if and only if  } x=a=1.\]
\end{theorem}

To prove this result, we need to overcome the difficulty that integers $a$ and $b$ are no longer symmetric in the equation $\binom{2k}{k} = \binom{2a}{a} \binom{x+2b}{b}$, unlike the case discussed in  \cite{Mor} and \cite{erdos}. The key tool used in our proof is an analysis of the existence of prime numbers in the product of consecutive integers, which was extensively investigated in \cite{Shorey1} and \cite{Shorey2}.

\section{Proof of Main Result}
To prove our main result, we will break the proof up into different cases. We will state these cases as propositions, and prove them throughout the paper by proving smaller lemmas. First, we will prove the result when $x\geq a$.

\begin{proposition}\label{prop2.1March24}
Let $k,x,b$ be positive integers and $a$ be a nonnegative integer such that $k=a+b$. Assume that $x\ge a$. Then
$$
\binom{2k}{k}\le \binom{2a}{a}\binom{x+2b}{b}.
$$ Moreover, the equality holds if and only if $x=a=1$.
\end{proposition}

\begin{proof}
Note that
$$\begin{aligned}
\frac { \binom{2a}{a}\binom{a+2b}{b}}{\binom{2k}{k}}& = \frac {(2a)!}{a!\cdot a!}\cdot \frac {(a+2b)!}{b!(a+b)!}\cdot \frac {k!\cdot k!}{(2k)!} \\
  & = \frac {(2a)(2a-1)\cdots(a+1)}{a (a-1)\cdots 1} \cdot \frac {k(k-1)\cdots(b+1)}{(2k)(2k-1)\cdots ( a+2b+1)}\\
  & = 2\cdot \frac {2a-1}{a-1} \frac {2a-2}{a-2}\cdots \frac {a+1}{1} \cdot \frac 1 2 \cdot \frac {k-1}{2k-1}\frac {k-2}{2k-2}\cdots \frac {k-a+1}{2k-a+1}\\
  &= \Big( \frac {2a-1}{a-1} \frac {k-1}{2k-1} \Big) \Big( \frac {2a-2}{a-2}\frac {k-2}{2k-2} \Big) \cdots \Big( \frac {a+1}{1}\frac {k-a+1}{2k-a+1} \Big) \\
  &  \left \{ \begin{aligned}
             & >1 \quad \text{ for } 1 < a <k\\
             & =1 \quad \text{ for } a=0,  1
               \end{aligned}
  \right .
\end{aligned}
$$
Now, if $x>a$, then $$\binom{2a}{a}\binom{x+2b}{b}>  \binom{2a}{a}\binom{a+2b}{b}\ge \binom{2k}{k}.$$
Furthermore, if $x=a>1$, then
$$\binom{2a}{a}\binom{x+2b}{b}= \binom{2a}{a}\binom{a+2b}{b}> \binom{2k}{k}.
$$ When $x=a=1$, we have
$$\binom{2a}{a}\binom{x+2b}{b}= \binom{2a}{a}\binom{a+2b}{b}= \binom{2}{1}\binom{1+2b}{b}= \binom{2k}{k}. $$
\end{proof}

Next, we will prove the result for $1\leq k \leq 10$.

\begin{lemma}\label{lemma2.2March24}
Let $k,x,a, b$ be positive integers such that $k=a+b$. Assume that $1\le k\le 10$. Then
$$
\binom{2k}{k}= \binom{2a}{a}\binom{x+2b}{b}.
$$ if and only if $x=a=1$.
\end{lemma}

\begin{proof} By Proposition \ref{prop2.1March24}, we might assume that $x\le a$. All cases for $1\le k\le 10$ and $x\le a$ are verified by a direct computation with Maple.
\end{proof}

In the following three lemmas, we will prove the result for $a-x=1,2,3$ when $4<2a\leq k$. We will fully prove Lemma \ref{lemma2.3March24}, and omit the proofs of Lemma \ref{lemma2.4March24} and Lemma \ref{lemma2.5March24} since they are similar.

\begin{lemma}\label{lemma2.3March24} Let $k,x,a, b$ be positive integers such that $k=a+b$. Then
$$
\binom{2k}{k}\ne \binom{2a}{a}\binom{a+2b-1}{b},\qquad \text { for }   \   4<2a\le  k.
$$
\end{lemma}

\begin{proof}

Let $k,x,a, b$ be positive integers such that $k=a+b$. Assume that 
\begin{equation}
\binom{2k}{k}=\binom{2a}{a}\binom{a+2b-1}{b} \ \  \text{ for }   4<2a\le  k. \label{equ2.1March24}
\end{equation}
%In other words,
%\begin{equation}
%\binom{2k}{k}=\binom{2a}{a}\binom{2k-a-1}{k-a} \ \  \text{ for }   4<2a\le  k. \label{equ2.1March24}
%\end{equation}
So,
$$
\frac {(2k)!}{(k!)^2}= \frac {(2a)!}{(a!)^2}\cdot \frac {(2k-a-1)!}{(k-a)!(k-1)!}.
$$Thus,
$$
\frac {2k (2k-1)(2k-2)\cdots (2k-a)}{k\cdot k \cdot (k-1) \cdots (k-a+1)}= \frac {(2a)(2a-1)\cdots (a+2)(a+1)}{a (a-1) \cdots 2\cdot 1}.
$$
We have, for $k>a$,
$$
\frac {2k-1} k < 2, \ \frac {2k-2}{k-1}< \frac {2a-1}{a-1}, \ \frac {2k-3}{k-2}< \frac {2a-2}{a-2}, \cdots, \frac{2k-a+1}{k-a+2}< \frac {a+2} 2
$$
and
$$
2 \cdot \frac{2k-a}{k-a+1}=2\Big( 2+ \frac {a-2}{k-a+1} \Big) <6\le  a+1 \qquad \text { for } a\ge 5.
$$ Thus, if $a\ge 5$, then
$$
\binom{2k}{k}<\binom{2a}{a}\binom{a+2b-1}{b} \qquad \text { for } k=a+b.
$$

When $a=4$, the assumption  (\ref{equ2.1March24})  becomes
$$
\binom{2k}{k}=\binom{8}{4}\binom{2k-5}{k-4},\qquad \text { for }    k\ge 8,
$$ or
$$
\frac {(2k)(2k-1)(2k-2)(2k-3)(2k-4) }{k\cdot k (k-1)(k-2)(k-3) }= \binom{8}{4}.
$$A simple computation shows that
$$
8(2k-1)(2k-3)=70k(k-3),
$$ which has a unique integer solution $k=4$. Thus
$$
\binom{2k}{k}\ne \binom{8}{4}\binom{2k-5}{k-4},\qquad \text { for }    k\ge 8.
$$

When $a=3$, the assumption   (\ref{equ2.1March24})  becomes
$$
\binom{2k}{k}=\binom{6}{3}\binom{2k-4}{k-3},\qquad \text { for }    k\ge 6.
$$ Equivalently, we have, $$
\frac {(2k)(2k-1)(2k-2)(2k-3)  }{k\cdot k (k-1)(k-2)  }=\binom{6}{3}.
$$ It induces that
$$
4(2k-1)(2k-3)=20k(k-2),
$$ where $k=-1, 3$ are integer solutions. Thus,
$$
\binom{2k}{k}\ne \binom{6}{3}\binom{2k-4}{k-3},\qquad \text { for }    k\ge 6.
$$
The proof is complete.
\end{proof}

\begin{lemma}\label{lemma2.4March24} %if you want this proof, it is in "Spring 2024 Paper" from April 4, 2024
Let $k,x,a, b$ be positive integers such that $k=a+b$. Then
$$
\binom{2k}{k}\ne \binom{2a}{a}\binom{a+2b-2}{b},\qquad \text { for }     4<2a\le  k.
$$
\end{lemma}

\begin{lemma}\label{lemma2.5March24} %if you want this proof, it is in "Spring 2024 Paper" from April 4, 2024
Let $k,x,a, b$ be positive integers such that $k=a+b$. Then
$$
\binom{2k}{k}\ne \binom{2a}{a}\binom{a+2b-3}{b},\qquad \text { for }    \   4<2a\le  k.
$$
\end{lemma}

Next, we will prove the main result in the case that $a\geq \frac{1}{2}k$ and $b\geq 3$. To do so, we will need the following results from Hanson \cite{Hans} and St\u anic\u a \cite{Stan}.

\begin{lemma}[Hanson \cite{Hans}] \label{lemma2.6March24}
The product of $m$ consecutive integers $n(n+1)\cdots(n+m-1)$ greater than $m$ contains a prime divisor greater that $\frac 3 2 m$ with the exceptions $3\cdot 4$, $8\cdot 9$, and $6\cdot 7\cdot 8\cdot 9 \cdot 10$.
\end{lemma}

\begin{lemma}[St\u anic\u a \cite{Stan}]\label{lemma2.7March24}
Let $m,n, r$ be positive integers, with $m>r\ge 1$ and $n\ge 1$. Then
$$
\frac 1 {\sqrt{2\pi}} e^{-\frac 1 {8n}} n^{-\frac 1 2}\frac {m^{mn+\frac 1 2}}{(m-r)^{(m-r)n+\frac 1 2}r^{rn+\frac 1 2}} <\binom{mn}{rn} <
   \frac 1 {\sqrt{2\pi}}   n^{-\frac 1 2}\frac {m^{mn+\frac 1 2}}{(m-r)^{(m-r)n+\frac 1 2}r^{rn+\frac 1 2}}
$$
\end{lemma}

Using the previous two results, we can prove the following Lemma.
\begin{lemma}\label{lemma2.8March24} Let $k,x,a, b$ be positive integers such that $k=a+b$.
If $a\ge \frac 1 2 k$, $b\ge 3$, then $$
\binom{2k}{k}\ne  \binom{2a}{a}\binom{x+2b}{b}.
$$
\end{lemma}
\begin{proof}We follows the strategy used in \cite{Mor}.  Suppose that $a\ge \frac 1 2 k$.
Since $b\ge 3$, we know $k\ge 6$. Assume $\binom{2k}{k}=  \binom{2a}{a}\binom{x+2b}{b}.$ Then
$$
\frac {\binom{2k}{k}}{ \binom{2a}{a}}=\frac {(2k)(2k-1)\cdots (2a+1)}{(k(k-1)\cdots (a+1))^2} = \binom{x+2b}{b}
$$ is an integer. By Lemma \ref{lemma2.6March24}, there exists a prime divisor $p$ of the product $(2k)(2k-1)\cdots (2a+1)$ such that $p>\frac 3 2 (2k-2a)=3b> 2.$ We claim that $p$ is not a divisor of $k(k-1)\cdots (a+1)$. In fact assume that $p$ divides $k-i$ for some $0\le i\le b-1$ and $\alpha$ is the largest positive integer such that $p^\alpha$ divides $k-i$. Then $\alpha$ is also the largest positive integer such that $p^\alpha $ divides $2(k-i)$, and $p$ is not a divisor of other terms $2k-j$, with $j\ne 2i$, in the numerator. In other words, $\alpha$ is the largest positive integer such that $p^\alpha$ divides $(2k)(2k-1)\cdots (2a+1)$. However, $p^{2\alpha}$ divides $(k(k-1)\cdots (a+1))^2$. This contradicts with the assumption that $
\frac {\binom{2k}{k}}{\binom{2a}{a}}$ is an integer. Hence $p$ doesn't divide the denominator  $k(k-1)\cdots (a+1)$.

So, we have found that $p$ divides $\binom{x+2b}{b}$. Since $p$ is prime, we have $x+2b\ge p>3b$. By Lemma \ref{lemma2.7March24}, we obtain, for $b\ge 3$,
$$
\binom{x+2b}{b}>  \binom{3b}{b}\ge \frac 1 {\sqrt{2\pi}} e^{-\frac 1 {8b}} b^{-\frac 1 2}\frac {3^{3b+\frac 1 2}}{2^{2b+\frac 1 2}} = \left ( \sqrt{\frac 3 {4\pi}}\cdot \frac 1{\sqrt b e^{\frac 1 {8b}}}\cdot \Big( \frac {27}{16}\Big)^b\right ) \cdot 4^b> 4^b.
$$
However,
$$
\frac {\binom{2k}{k}}{\binom{2a}{a}}=\frac {(2k)(2k-1)\cdots (2a+1)}{(k(k-1)\cdots (a+1))^2}=2^b\cdot \frac {(2k-1)(2k-3)\cdots (2a+1)}{k(k-1)\cdots (a+1)} < 4^b.
$$This is a contradiction. Thus, for $b\ge 3$,
$$\binom{2k}{k}\ne \binom{2a}{a}\binom{x+2b}{b}.
$$\end{proof}

Using the previous lemma, we can now prove our main result when $a\geq \frac{1}{2}k$ and $x\geq 2$. 

\begin{lemma}\label{lemma2.9March24} Let $k,x,a, b$ be positive integers such that $k=a+b$.
If $a\ge \frac 1 2 k$  and $x\ge 2$, then $$
\binom{2k}{k}\ne \binom{2a}{a}\binom{x+2b}{b}.
$$
\end{lemma}
\begin{proof}
By Lemma \ref{lemma2.8March24}, we might assume that $1\le b<3$.   When $b=1$,  for $x\ge 2$,
$$\begin{aligned}
\binom{2(k-1)}{k-1}\binom{x+2}{1} &= \frac {(2(k-1))!}{((k-1)!)^2}(x+2) \\ &=\binom{2k}{k} (x+2)\cdot \frac {k^2}{2k(2k-1)} \\ &>\binom{2k}{k}. \end{aligned}
$$

When $b=2$, we have
$$ \frac{\binom{2(k-2)}{k-2}\binom{x+4}{2}}{\binom{2k}{k}} = \frac{(x+4)(x+3)}{2} \frac{k^2 (k-1)^2}{(2k)(2k-1)(2k-2)(2k-3)}.$$
If $x\ge 3$,
$$\begin{aligned}
\frac {(x+4)(x+3)}2\cdot \frac {k^2(k-1)^2}{2k(2k-1)(2k-2)(2k-3)} &=\frac {(x+4)(x+3)}8\cdot \frac {k (k-1) }{  (2k-1) (2k-3)} \\& >\frac {(x+4)(x+3)}{32} \\& >1. \end{aligned}$$
If $x=2$,
$$
\frac {(2+4)(2+3)}2\cdot\frac {k^2(k-1)^2}{2k(2k-1)(2k-2)(2k-3)}=1
$$ has no integer solution. Thus, when $b=2$, for all $x\ge 2$,
$$
\binom{2(k-2)}{k-2}\binom{x+4}{2} \ne \binom{2k}{k}.
$$
\end{proof}

We will now prove our main result when $x=1$.

\begin{proposition}\label{prop2.10March24}
Let $k,a, b$ be positive integers such that $k=a+b$. Then
$$
\binom{2k}{k} = \binom{2a}{a}\binom{1+2b}{b} \quad \text { if and only if } a=1.
$$
\end{proposition}
\begin{proof}
First, assume $a=1$. Then, 
\begin{align*}
  \frac{\binom{2k}{k}}{\binom{1+2b}{b}} &= \frac{(2k)! \cdot b! \cdot (b+1)!}{(1+2b)! \cdot k! \cdot k!} \\
  &= \frac{(2+2b)(1+2b)\dots (b+2)}{(1+2b)(2b)\dots (b+2)} \cdot \frac{b(b-1)\dots 1}{(b+1)b(b-1) \dots 1} \\
  &= \frac{2(1+b)}{1+b} \\
  &=2 \\
  &= \binom{2a}{a}
\end{align*}

Next, assume $\binom{2k}{k} = \binom{2a}{a}\binom{1+2b}{b}$. By Lemma \ref{lemma2.2March24}, we might assume that $k>10$.   Lemma \ref{lemma2.8March24} implies that, if $a\ge \frac 1 2 k$ and $b\ge 3$, then $
\binom{2k}{k} \ne \binom{2a}{a}\binom{1+2b}{b}$. Moreover, when $b=1, 2$, the  equation
$$
\binom{2k}{k} = \binom{2a}{a}\binom{1+2b}{b}
$$ has no integer solution for $k>10$. Therefore, we might assume that $a<\frac 1 2 k$, so $b>\frac 1 2 k$. For the purpose of the contradiction, we  further assume that $a>1$.

Then
$$\begin{aligned}
\frac {\binom{2k}{k}}{\binom{1+2b}{b}}&=\frac {(2k)(2k-1)\cdots (2b+2)}{(k(k-1)\cdots (b+1))(k(k-1)\cdots (b+2)) } \\& =\frac {(2k)(2k-1)\cdots (2b+3) \cdot 2}{(k(k-1)\cdots (b+2))^2  }  \\&= \binom{2a}{a}\end{aligned}
$$ is an integer. By Lemma \ref{lemma2.6March24}, there exists a prime divisor $p$ of the product $(2k)(2k-1)\cdots (2b+3)$ such that $p>\frac 3 2 (2k-2b-2)=3a -  3  > 2.$ We claim that $p$ is not a divisor of $k(k-1)\cdots (b+2)$. In fact assume that $p$ divides $k-i$ for some $0\le i\le a-1$ and $\alpha$ is the largest positive integer such that $p^\alpha$ divides $k-i$. Then $\alpha$ is also the largest positive integer such that $p^\alpha $ divides $2(k-i)$, and $p$ is not a divisor of other terms $2k-j$, with $j\ne 2i$, in the numerator. In other words, $\alpha$ is the largest positive integer such that $p^\alpha$ divides $(2k)(2k-1)\cdots (2b+3)\cdot 2$. However, $p^{2\alpha}$ divides $(k(k-1)\cdots (b+2))^2$. This contradicts with the assumption that $
\frac {\binom{2k}{k}}{\binom{1+2b}{b}}$ is an integer. Hence $p$ doesn't divide the denominator  $k(k-1)\cdots (b+2)$.

So, we must have $p$ divides $\binom{2a}{a}$, whence $ 2a\ge p>3a -  3  $. So $a<  3  $.  From the  assumption that $a>1$, we have $a=2$. Now that equation $
\binom{2k}{k} = \binom{2a}{a}\binom{1+2b}{b}
$ becomes
$$
\binom{2k}{k} = \binom{4}{2}\binom{2k-3}{k-2}, \qquad \text { when } a=2.
$$Equivalently, we get,
$$
\frac {(2k)(2k-1)(2k-2)}{k\cdot k(k-1)}=6,
$$ which  has no integer solution when $k>10$, a contradiction.

Therefore, $a\le 1$. Since $a$ is a positive integer, it follows that $a=1$. This completes the proof of the result.
\end{proof}

For the next Lemma we will prove, we need the following results from \cite {Shorey1} and \cite {Shorey2}.

\begin{lemma}[Shanta \& Shorey \cite{Shorey1}]\label{lemma2.11March24}
The product of $m$ consecutive integers $n(n+1)\cdots(n+m-1)$  contains a prime divisor greater than $1.8 m$ if  $n>m>2$ and $n+m\ge 150$.
\end{lemma}

\begin{lemma}[Nair \& Shorey \cite{Shorey2}]\label{lemma2.12March24}
The product of $m$ consecutive integers $n(n+1)\cdots(n+m-1)$  contains a prime divisor greater than $4.42 m$ if  $n>4m$, $m> 3$ and $n+m\ge 150$.
\end{lemma}

\begin{lemma}\label{lemma2.13March24}
 Let $k,x,a,b$ be positive integers such that  $k=a+b\ge 75$ and $x\ge 2$. If $a\le 0.9k$ and $1\le b\le 0.8k$, then
  $$
\binom{2k}{k}\ne \binom{2a}{a}\binom{x+2b}{b}.
$$
\end{lemma}
\begin{proof}Assume that $a\le 0.9k$ and $1\le b\le 0.8k$. Assume, by means of contradiction, that $
\binom{2k}{k}= \binom{2a}{a}\binom{x+2b}{b}.
$
By Lemma \ref{lemma2.11March24}, there exists a prime divisor of the product $(2k)(2k-1)\cdots (k+1)$ such that $p>1.8k$. Obviously, $p$ is also a divisor of $
\binom{2k}{k}$, thus a divisor of $\binom{2a}{a}\binom{x+2b}{b}$. It follows that $2a>1.8k$ or $x+2b>1.8k$. By Proposition \ref{prop2.1March24}, we might assume that $x<a$. Thus, if $x+2b>1.8k$, then $ b=a+2b-(a+b)>x+2b-k>0.8k$. Therefore, we conclude that either $a>0.9k$ or $b>0.8k$, a contradiction to $a\le 0.9k$ and $1\le b\le 0.8k$. Thus, 
$$
\binom{2k}{k}\ne \binom{2a}{a}\binom{x+2b}{b}.
$$
\end{proof}

\begin{lemma}\label{lemma2.14March24}
 Let $k,x,a,b$ be positive integers such that $x\ge 2$ and $k=a+b\ge 150$. If $a\ge \frac {171}{121} x$, then $$
\binom{2k}{k}\ne \binom{2a}{a}\binom{x+2b}{b} .
$$
\end{lemma}
\begin{proof}
Assume that $x\ge  2$,   $a\ge \frac {171}{121} x$ and  $k\ge 150$. Assume, by means of contradiction, that $
\binom{2k}{k}= \binom{2a}{a}\binom{x+2b}{b}$ for some $ b\ge 1.$ By Lemma \ref{lemma2.9March24}, we can assume that $a < \frac 1 2k$. By Lemma \ref{lemma2.13March24}, we can assume that $b>0.8k$, thus $a=k-b< 0.2k$. In particular, $b>4a$. By Lemma \ref{lemma2.3March24}, Lemma \ref{lemma2.4March24}, and Lemma \ref{lemma2.5March24} we can assume that $a-x>3$.
So,
$$
 \binom{2a}{a}=\frac {\binom{2k}{k}}{\binom{x+2b}{b}} =\frac {(2k)(2k-1)\cdots (x+2b+1)}{(k \cdot (k-1)\cdots (b+1))(k\cdot (k-1)\cdots (x+b+1))} .
$$ Then
$$
Q=\frac {(2k)(2k-1)\cdots (x+2b+1)}{(k \cdot (k-1)\cdots (x+b+1))^2}  = \binom{2a}{a}\cdot (b+1)\cdots (b+x)
$$ is an integer. Let $n=b+x+1$ and $m=a-x$. Since $k-(x+b)=m> 3 $, we have, $n>b>4a>4m$, and $n+m=k+1 \ge 150$. So, by Lemma \ref{lemma2.12March24}, there exists a prime divisor $p$ of $k \cdot (k-1)\cdots (x+b+1)$ such that $p>4.42(a-x)$. Thus,  for some $0\le i<a-x$, $p$ divides  $k-i$ in the denominator $k\cdot (k-1)\cdots (x+b+1)$ of $Q$. Let $\alpha$ be the largest positive integer such that $p^\alpha$ divides  $k-i$. Then $\alpha$ is also the largest positive integer such that $p^\alpha$ divides $2(k-i)$, and no other term $(2k-t)$, %$0\leq t<2a-x$, 
for $t\neq 2i$, in the numerator $(2k)(2k-1)\cdots (x+2b+1)$ of $Q$ is divisible by  $p$ as $a\ge \frac {171}{121} x$ and $p>4.42(a-x)\ge (2a-x)$. %WHY IS THIS
However $p^{2\alpha}$ is a divisor of the denominator $(k \cdot (k-1)\cdots (x+b+1))^2$ of $Q$. That contradicts with the fact that $Q$ is an integer. Therefore, if $x\ge  2$,   $a\ge \frac {171}{121}x$ and  $k\ge 150$, then $$
\binom{2k}{k}\ne \binom{2a}{a}\binom{x+2b}{b} .
$$\end{proof}

Using the previous lemmas, we can prove our main result when $x\geq 2$ and $k \geq 150$.

\begin{proposition}\label{prop2.15March24}
Let $k,x,a,b$ be positive integers such $x\ge 2$ and $k=a+b\ge 150$. Then $$
\binom{2k}{k}\ne \binom{2a}{a}\binom{x+2b}{b}.
$$
\end{proposition}

\begin{proof}
Assume, by means of contradiction, that $\binom{2k}{k}= \binom{2a}{a}\binom{x+2b}{b}$ for some $b\ge 1.$ By Proposition \ref{prop2.1March24} and Lemma \ref{lemma2.14March24}, we might assume that $x<a<\frac{171}{121} x<\frac 3 2 x$. By Lemma \ref{lemma2.13March24}, we assume that $b>0.8k$, which gives us $a < 0.2k$. Note that $$
\binom{x+2b}{b} =\frac {(x+2b)!}{b! \cdot (x+b)!} = \frac {(2b)!}{b!\cdot b!}\cdot \frac {(2b+1)(2b+2)\cdots (2b+x)}{(b+1)(b+2)\cdots (b+x)}.
$$
By Lemma \ref{lemma2.7March24} and the assumption that $a<\frac  3 2 x$,
$$\begin{aligned}
\frac {\binom{2k}{k}}{\binom{2a}{a}\binom{2b}{b}} & \le \frac {\frac 1 {\sqrt{2\pi }}k^{-\frac 1 2}2^{2k+\frac 1 2}}{
\frac 1 {\sqrt{2\pi }}e^{-\frac 1 {8a}}a^{-\frac 1 2}2^{2a+\frac 1 2}\cdot \frac 1 {\sqrt{2\pi }}e^{-\frac 1{8b}}b^{-\frac 1 2}2^{2b+\frac 1 2}  } \\
&= e^{\frac 1 {8a}+\frac 1 {8b}}\sqrt{\frac {\pi ab} k} \\
&< e^{1/4}\sqrt{\frac{3\pi x}{2} }.
\end{aligned}$$
On the other hand,
$$\begin{aligned}
\frac {(2b+1)(2b+2)\cdots (2b+x)}{(b+1)(b+2)\cdots (b+x)} & = \Big(2-\frac 1 {b+1}\Big) \Big(2- \frac 2 {b+2} \Big) \cdots \Big( 2- \frac x {b+x} \Big) \\
&> \Big( 2- \frac x {b+x} \Big)^x > \Big( 2- \frac a {b+a} \Big)^x \ge \Big( 2-  \frac 1 5  \Big)^x.  \end{aligned}
$$Note that $  \big( \frac 9 5  \big)^x -  e^{1/4}\sqrt{3\pi x/2 }$ is an increasing function for $x\ge 2$ and
$  \big( \frac 9 5  \big)^3 > e^{1/4}\sqrt{9\pi  /2 }$. So, for $x\ge 3$ and $k\ge 150$,
$$
\frac {(2b+1)(2b+2)\cdots (2b+x)}{(b+1)(b+2)\cdots (b+x)} > \Big(\frac{9}{5} \Big)^x > e^{1/4}\sqrt{\frac{3\pi x}{2}} > \frac {\binom{2k}{k}}{\binom{2a}{a}\binom{2b}{b}}
$$
which gives us
$$
\binom{x+2b}{b} = \binom{2b}{b} \cdot \frac {(2b+1)(2b+2)\cdots (2b+x)}{(b+1)(b+2)\cdots (b+x)} > \frac {\binom{2k}{k}}{\binom{2a}{a}}
$$
and therefore,
$$
\binom{2a}{a}\binom{x+2b}{b} > \binom{2k}{k}
$$
which is a contradiction to our assumption.
Also, we know that $x<a<\frac{3}{2}x$, so when $x=2$, we must have $2<a<3$, a contradiction to $a$ an integer. Therefore, for $x\geq 2$ and $k \geq 150$,
 $$
\binom{2k}{k}\ne \binom{2a}{a}\binom{x+2b}{b}.
$$
\end{proof}

Finally, we will prove our main result when $x\geq 2$ and $x<k< 150$.

\begin{proposition}\label{prop2.16March24}
Let $k,x,a,b$ be positive integers such $x\ge 2$, $k=a+b$, and $x< k< 150$. Then 
$$
\binom{2k}{k}\ne \binom{2a}{a}\binom{x+2b}{b}.
$$
\end{proposition}
\begin{proof}
Using Maple, we verified that, for $x\ge 2$ and $x< k <  150$,
$$
\binom{2k}{k}\ne \binom{2a}{a}\binom{x+2k}{k}.
$$
\end{proof}

Now, we can prove the following, main theorem of the paper.

{
\renewcommand{\thetheorem}{\ref{Thm1.1}}
\begin{theorem}
  Let $k,x, b$ be positive integers and $a$ be a nonnegative integer such that \\
\indent $k=a+b$. Then
$$
\binom{2k}{k}= \binom{2a}{a}\binom{x+2b}{b}, \qquad \text { if and only if  } x=a=1.
$$
\end{theorem}
\addtocounter{theorem}{-1}
}

\begin{proof}
  When $x=1$, we obtain the result from Proposition \ref{prop2.10March24}.

When $x\geq 2$, and $k\geq 150$, we obtain the result from Proposition \ref{prop2.15March24}.

When $x\geq 2$, and $x<k\leq 150$, we obtain the result from Proposition \ref{prop2.16March24}.

When $x\geq 2$, and $x\geq k$, we obtain the result from Proposition \ref{prop2.1March24}.
\end{proof}

\end{document}